\documentclass[]{interact}

\usepackage{epstopdf}
\usepackage[caption=false]{subfig}

\usepackage[numbers,sort&compress]{natbib}
\bibpunct[, ]{[}{]}{,}{n}{,}{,}
\makeatletter
\def\NAT@def@citea{\def\@citea{\NAT@separator}}
\makeatother

\theoremstyle{plain}
\newtheorem{theorem}{Theorem}[section]

\newtheorem{proposition}[theorem]{Proposition}

\theoremstyle{definition}
\newtheorem{definition}[theorem]{Definition}
\newtheorem{example}[theorem]{Example}

\theoremstyle{remark}
\newtheorem{remark}[theorem]{Remark}

\usepackage{graphicx,color,enumerate,comment}
\usepackage{tikz}
\usepackage{pgfplots}
\begin{document}
\title{On approximate quasi Pareto solutions in nonsmooth semi-infinite interval-valued vector optimization problems}

\author{
\name{Nguyen Huy Hung\textsuperscript{a}, Hoang Ngoc Tuan\textsuperscript{a}, and Nguyen Van Tuyen\textsuperscript{a}\thanks{CONTACT Nguyen Van Tuyen. Email: nguyenvantuyen83@hpu2.edu.vn; tuyensp2@yahoo.com}}
\affil{\textsuperscript{a}Department of Mathematics, Hanoi Pedagogical University 2, Xuan Hoa, Phuc Yen, Vinh Phuc, Vietnam}
}

\maketitle

\begin{abstract}
This paper deals with approximate solutions of a nonsmooth semi-infinite programming with multiple interval-valued objective functions.  We first introduce four types of approximate quasi Pareto solutions of the considered problem by considering the lower-upper interval order relation and then apply some advanced tools of variational analysis and generalized differentiation to establish necessary optimality conditions for these approximate solutions. Sufficient conditions for approximate quasi Pareto solutions of such a problem are also provided by means of introducing the concepts of approximate (strictly) pseudo-quasi generalized convex functions defined in terms of the limiting subdifferential of locally Lipschitz functions. Finally, a Mond--Weir type dual model in approximate form is formulated, and weak, strong and converse-like duality relations are proposed. 	
\end{abstract}

\begin{amscode}
90C29; 90C46; 90C70; 90C34;  49J52	
\end{amscode}

\section{Introduction}
\label{intro}
In this paper, we are interested in approximate solutions of the following semi-infinite programming with multiple interval-valued objective functions:  
\begin{align}\label{problem}\tag{SIVP}
LU-&\mathrm{Min}\, f(x):=(f_1(x), \ldots, f_m(x)) 
\\
&\text{s. t.}\ \ x\in \mathcal{F}:=\{x\in \Omega\,:\, g_t(x)\leq 0, t\in T\}, \notag
\end{align}
where $f_i\colon \mathbb{R}^n\to \mathcal{K}_c$, $i\in I:=\{1, \ldots, m\}$, are interval-valued functions defined by $f_i(x)=[f_i^L(x), f_i^U(x)]$, $f_i^L,$ $f_i^U\colon\mathbb{R}^n\to \mathbb{R}$  are locally Lipschitz functions satisfying $f_i^L(x)\leq f_i^U(x)$ for all $x\in\mathbb{R}^n$ and $i\in I$, $\mathcal{K}_c$ is the class of all closed and bounded intervals in $\mathbb{R}$, i.e., 
$$\mathcal{K}_c=\{[a^L, a^U]\,:\, a^L, a^U\in\mathbb{R}, a^L\leq a^U\},$$ 
$g_t\colon\mathbb{R}^n\to \mathbb{R}$, $t\in T$, are locally Lipschitz functions, $T$ is an arbitrary  set (possibly infinite), and $\Omega$ is a nonempty and closed subset of $\mathbb{R}^n$. Set $g_T:=(g_t)_{t\in T}$.

An interval-valued optimization problem is one of the deterministic optimization models to deal with the uncertain (incomplete) data. In the literature, there are three main approaches to model constrained optimization with uncertainty, say {\em stochastic programming approach}, {\em fuzzy  programming approach}, and {\em interval-valued programming approach}; see, e.g., \cite{Ben-Tal-Nemirovski-08,Ben-Tal-Nemirovski-09,Ben-Tal-Nemirovski-98,Slowinski-98,Wu-09-c,Rahimi}. Many methodologies have been developed to solve these problems. However, it should be noted here that the usual way is to transform stochastic and fuzzy optimization problems into the conventional optimization problems; frequently, these problems are very complicated. Consequently, stochastic and fuzzy optimization problems are not easy to be solved. In interval-valued optimization, the coefficients of objective and constraint functions are taken as closed intervals. Hence, the interval-valued optimization problem will be easier to be solved than a stochastic or fuzzy optimization one. That is the main reason why the interval-valued optimization problems have recently received increasing interest in optimization community; see, e.g., \cite{Chalco-Cano et.al.-13,Ishibuchi-Tanaka-90,Jennane,Kumar,Osuna-Gomez-17,Qian,Singh-Dar-15,Singh-Dar-Kim-16,Singh-Dar-Kim-19,Su-2020,Tung-2019,Tung-2019b,Tuyen-2021,Wu-07,Wu-09,Wu-09-b,Wu-09-c} and the references therein.  

In \cite{Ishibuchi-Tanaka-90}, Ishibuchi and Tanaka  introduced the lower-upper (LU) interval order relation and reformulated optimization problems with interval-valued objective functions as vector optimization problems using the order relation.  Thereafter, many authors have  investigated optimality conditions  of Karush--Kuhn-Tucker-type (KKT) and duality for optimization problems with one or multiple interval-valued objective functions  and finitely many  constraints; see, e.g., \cite{Chalco-Cano et.al.-13,Osuna-Gomez-17,Singh-Dar-15,Singh-Dar-Kim-16,Singh-Dar-Kim-19,Su-2020,Tuyen-2021,Wu-07,Wu-09,Wu-09-b,Wu-09-c}. 

The semi-infinite optimization problems play a very important role in optimization theory and their models cover optimal control, approximation theory,  semi-definite programming and numerous engineering problems, etc. However, in contrast with the case of optimization problems with finitely many  constraints, there have only been several papers dealing with KKT optimality conditions and duality for semi-infinite interval-valued optimization problems. For papers of this topic, we refer the reader to \cite{Jennane,Kumar,Qian,Tung-2019,Tung-2019b} and references given therein. In \cite{Kumar}, Kumar et al. established optimality conditions and duality theorems for  interval-valued programming problems with infinitely many constraints. Then, in \cite{Jennane,Qian,Tung-2019,Tung-2019b}, the authors presented optimality conditions and duality theorems for Pareto optimal solutions with respect to $LU$ interval order relation of a semi-infinite interval-valued vector optimization problem   under the convexity of the objective functions and constraints. To the best of our
knowledge, so far there have been no papers investigating optimality conditions and duality relations for approximate Pareto solutions of such a problem.  It should be noted here that the study of approximate solutions is very important in optimization  because, from the computational point of view, numerical algorithms usually generate only approximate solutions if we stop them after a finite number of steps. Furthermore, the solution set may be empty in the general noncompact case, whereas approximate solutions exist under very weak assumptions; see e.g.,  \cite{Son-Tuyen-Wen-19,Bao-et al,Tuyen-2021,Loridan-84,Chuong-Kim-16,TXS-20,Kim-Son-18,Chen13,Jahn-2011,Luc-89}. 

Motivated by the above observations, in this paper, we introduce four kinds of approximate quasi Pareto solutions with respect to $LU$ interval order relation for problems of the form \eqref{problem}. Then we employ the Mordukhovich/limiting subdifferential and the Mordukhovich/limiting normal cone (cf. \cite{mor06}) to examine  KKT optimality conditions and duality relations for these approximate solutions of problem \eqref{problem}. 

The paper is organized as follows. Section \ref{Preliminaries} provides some basic definitions from variational analysis, interval analysis and several auxiliary results. In Section \ref{Optimiality-conditions}, we introduce four kinds of approximate quasi Pareto solutions of problem \eqref{problem} and establish KKT-type necessary conditions for these approximate solutions.  With the help of approximate generalized convex functions defined in terms of the Mordukhovich/limiting subdifferential and the Mordukhovich normal cone, we provide sufficient conditions for approximate quasi Pareto solutions of the considered problem. Section \ref{Duality-Relations} is devoted to presenting duality relations for approximate quasi Pareto solutions. The conclusions
are presented in  the final section.

\section{ Preliminaries} \label{Preliminaries}
We use the following notation and terminology. Fix $n \in {\mathbb{N}}:=\{1, 2, \ldots\}$. The space $\mathbb{R}^n$ is equipped with the usual scalar product and  Euclidean norm. The closed unit ball of $\mathbb{R}^n$ is denoted by $\mathbb{B}_{\mathbb{R}^n}$.  We denote the nonnegative orthant in $\mathbb{R}^n$ by  $\mathbb{R}^n_+$.  The topological closure is denoted  by  $\mathrm{cl}\,{S}$.

\begin{definition}[{see \cite{mor06}}]{\rm  Given $\bar x\in  \mbox{cl}\,S$. The set
		\begin{equation*}
		N(\bar x; S):=\{z^*\in \mathbb{R}^n:\exists
		x^k\stackrel{S}\longrightarrow \bar x, \varepsilon_k\to 0^+, z^*_k\to z^*,
		z^*_k\in {\widehat N_{\varepsilon
				_k}}(x^k; S),\ \ \forall k \in\mathbb{N}\},
		\end{equation*}
		is called the {\em Mordukhovich/limiting normal cone}  of $S$ at $\bar x$, where
		\begin{equation*}
		\widehat N_\varepsilon  (x; S):= \bigg\{ {z^*  \in {\mathbb{R}^n} \;:\;\limsup_{u\overset{S} \rightarrow x}
			\frac{{\langle z^* , u - x\rangle }}{{\parallel u - x\parallel }} \leq \varepsilon } \bigg\}
		\end{equation*}
		is the set of  {\em $\varepsilon$-normals} of $S$  at $x$ and  $u\xrightarrow {{S}} x$ means that $u \rightarrow x$ and $u \in S$.
	}
\end{definition}

\medskip
Let  $\varphi \colon \mathbb{R}^n \to  \overline{\mathbb{R}}$ be an {\em extended-real-valued function}. The {\em  epigraph}  and  {\em domain} of $\varphi$ are denoted, respectively, by
\begin{align*}
\mbox{epi }\varphi&:=\{(x, \alpha)\in\mathbb{R}^n\times\mathbb{R} \,:\,  \alpha\geq \varphi(x) \},
\\
\mbox{dom }\varphi &:= \{x\in \mathbb{R}^n \,:\, \ \ |\varphi(x)|<+\infty \}.
\end{align*}

\begin{definition}[{see \cite{mor06}}]{\rm   
		Let $\bar x\in \mbox{dom }\varphi$. The set
		\begin{align*}
		\partial \varphi (\bar x):=\{x^*\in \mathbb{R}^n \,:\, (x^*, -1)\in N((\bar x, \varphi (\bar x)); \mbox{epi }\varphi )\},
		\end{align*}
		is called the {\it Mordukhovich/limiting subdifferential}  of $\varphi$ at $\bar x$. If $\bar x\notin \mbox{dom }\varphi$, then we put $\partial \varphi (\bar x)=\emptyset$.
	}
\end{definition}

We now summarize some properties of  the Mordukhovich subdifferential that will be used in the next section.
\begin{proposition}[{see \cite[Theorem 3.36]{mor06}}]\label{sum-rule} Let $\varphi_l\colon\mathbb{R}^n\to\overline{\mathbb{R}}$, $l=1, \ldots, p$, $p\geq 2$, be lower semicontinuous around $\bar x$ and let all but one of these
	functions be locally Lipschitz around $\bar x$. Then we have the following inclusion
	\begin{equation*}
	\partial (\varphi_1+\ldots+\varphi_p) (\bar x)\subset \partial  \varphi_1 (\bar x) +\ldots+\partial \varphi_p (\bar x).
	\end{equation*}
\end{proposition}

\begin{proposition}[{see \cite[Theorem 3.46]{mor06}}]\label{max-rule}
	Let $\varphi_l\colon\mathbb{R}^n\to\overline{\mathbb{R}}$, $l=1, \ldots, p$,  be  locally Lipschitz around $\bar x$. Then the function
	$ \phi(\cdot):=\max\{\varphi_l(\cdot):l=1, \ldots, p\}$
	is also locally Lipschitz around $\bar x$ and one has
	\begin{equation*}
	\partial \phi(\bar x)\subset \bigcup\bigg\{\partial\bigg(\sum_{l=1}^{p}\lambda_l\varphi_l\bigg)(\bar x)\;:\; (\lambda_1, \ldots, \lambda_p)\in\Lambda(\bar x)\bigg\},
	\end{equation*}
	where $\Lambda(\bar x):=\big\{(\lambda_1, \ldots, \lambda_p)\;:\; \lambda_l\geq 0, \sum_{l=1}^{p}\lambda_l=1, \lambda_l[\varphi_l(\bar x)-\phi(\bar x)]=0\big\}.$
\end{proposition}

\begin{proposition}[{\rm see \cite[Proposition 1.114]{mor06}}] \label{Fermat-rule} Let   $\varphi\colon\mathbb{R}^n\to\overline{\mathbb{R}}$  be finite at $\bar x$. If $\varphi$ has a local minimum at $\bar x$, then $ 0\in\partial\varphi(\bar x).$
\end{proposition}

Next we recall some definitions and facts in interval analysis. Let $A=[a^L, a^U]$ and $B=[b^L, b^U]$ be two intervals in $\mathcal{K}_c$. Then, by definition, we have
\begin{enumerate}[(i)]
	\item $A+B:=\{a+b\,:\, a\in A, b\in B\}=[a^L+b^L, a^U+b^U]$;
	\item $A-B:=\{a-b\,:\, a\in A, b\in B\}=[a^L-b^U, a^U-b^L]$;
	\item For each $k\in\mathbb{R}$,
	\begin{equation*}kA:=\{ka\,:\, a\in A\}=
	\begin{cases}
	[ka^L, ka^U]\ \ \text{if}\ \ k\geq 0,
	\\
	[ka^U, ka^L]\ \ \text{if}\ \ k < 0, 
	\end{cases}
	\end{equation*}
\end{enumerate}
see, e.g., \cite{Alefeld,Moore-1966,Moore-1979}  for more details.
It should be note that if $a^L=a^U$, then $A=[a, a]=a$ is a real number.

\begin{definition}[{see \cite{Ishibuchi-Tanaka-90,Wu-07}}]{\rm 
		Let $A=[a^L, a^U]$ and $B=[b^L, b^U]$ be two intervals in $\mathcal{K}_c$. We say that:
		\begin{enumerate}[(i)]
			\item $A\leq_{LU} B$ if $a^L\leq b^L$ and $a^U\leq b^U$.
			
			\item  $A<_{LU} B$ if $A\leq_{LU} B$ and $A\neq B$, or, equivalently, $A<_{LU} B$ if
			\\
			$ 
			\begin{cases}
			a^L<b^L
			\\
			a^U\leq b^U,
			\end{cases}
			$ 
			or \ \ \ \
			$ 
			\begin{cases}
			a^L\leq b^L
			\\
			a^U< b^U,
			\end{cases}
			$ 
			or \ \ \ \
			$ 
			\begin{cases}
			a^L< b^L
			\\
			a^U< b^U.
			\end{cases}
			$ 
			\item $A<^s_{LU} B$ if $a^L<b^L$ and $a^U<b^U$.
		\end{enumerate}
		We also define $A\leq_{LU} B$ (resp., $A<_{LU} B$, $A<^s_{LU} B$) if and only if 	$B\geq_{LU} A$ (resp., $B>_{LU} A$, $B>^s_{LU} A$).}
\end{definition}

\section{Approximate optimality conditions}\label{Optimiality-conditions}
Let $\mathbb{R}_+^{|T|}$ denote the set of all functions $\mu\colon T\to\mathbb{R}_+$ taking values $\mu_t:=\mu(t)=0$ for all $t\in T$ except for finitely many points. The active constraint multipliers set at $\bar x\in \Omega$ is defined by 
\begin{equation*}
A(\bar x):=\left\{\mu\in \mathbb{R}_+^{|T|}\;:\; \mu_tg_t(\bar x)=0, \ \ \forall t\in T\right\}.
\end{equation*}
For each $\mu\in A(\bar x)$, put $T(\mu):=\{t\in T\;:\; \mu_t\neq 0\}$.

We now introduce approximate solutions of \eqref{problem} with respect to $LU$ interval order relation.  Let $\epsilon_i^L$,  $\epsilon_i^U$, $i\in I$, be   real numbers satisfying $0\leq \epsilon_i^L\leq \epsilon_i^U$ for all $i\in I$  and put   $\mathcal{E}:=(\mathcal{E}_1, \ldots, \mathcal{E}_m)$, where $\mathcal{E}_i:=[\epsilon_i^L, \epsilon_i^U]$.
\begin{definition}\label{Defi-solution}{\rm
		Let $\bar x\in \mathcal{F}$. We say that:
		\begin{enumerate}[(i)]
			\item $\bar x$ is a {\em type-1 $\mathcal{E}$-quasi Pareto solution} of \eqref{problem}, denoted by $\bar x\in \mathcal{E}$-$\mathcal{S}_1^q\eqref{problem}$, if there is no $x\in \mathcal{F}$ such that
			\begin{equation*}
			\begin{cases}
			f_i(x)\leq_{LU} f_i(\bar x)-\mathcal{E}_i\|x-\bar x\|,\ \ &\forall i\in I,
			\\
			f_k(x)<_{LU} f_k(\bar x)-\mathcal{E}_k\|x-\bar x\|,\ \ &\text{for at least one}\ \ k\in I.
			\end{cases} 
			\end{equation*}
			
			\item $\bar x$ is a {\em type-2 $\mathcal{E}$-quasi Pareto solution} of \eqref{problem}, denoted by $\bar x\in \mathcal{E}$-$\mathcal{S}_2^q\eqref{problem}$, if there is no $x\in \mathcal{F}$ such that
			\begin{equation*}
			\begin{cases}
			f_i(x)\leq_{LU} f_i(\bar x)-\mathcal{E}_i\|x-\bar x\|,\ \ &\forall i\in I,
			\\
			f_k(x)<^s_{LU} f_k(\bar x)-\mathcal{E}_k\|x-\bar x\|,\ \ &\text{for at least one}\ \ k\in I.
			\end{cases} 
			\end{equation*}
			
			\item  $\bar x$ is a {\em type-1 $\mathcal{E}$-quasi-weakly Pareto solution} of \eqref{problem}, denoted by $\bar x\in \mathcal{E}$-$\mathcal{S}_1^{qw}\eqref{problem}$, if there is no $x\in \mathcal{F}$ such that
			\begin{equation*}
			f_i(x)<_{LU} f_i(\bar x)-\mathcal{E}_i\|x-\bar x\|,\ \ \forall i\in I.
			\end{equation*}
			
			\item $\bar x$ is a {\em type-2 $\mathcal{E}$-quasi-weakly Pareto solution} of \eqref{problem}, denoted by $\bar x\in \mathcal{E}$-$\mathcal{S}_2^{qw}\eqref{problem}$, if there is no $x\in \mathcal{F}$ such that
			\begin{equation*}
			f_i(x)<^s_{LU} f_i(\bar x)-\mathcal{E}_i\|x-\bar x\|,\ \ \forall i\in I.
			\end{equation*}    
		\end{enumerate}
	}
\end{definition}

It should be note that, if $\mathcal{E}=0$, i.e., $\epsilon^L_i=\epsilon^U_i=0$, $i\in I$, then the notion of a type-1 $\mathcal{E}$-quasi Pareto solution (resp.,  a type-2 $\mathcal{E}$-quasi Pareto solution,  a type-1 $\mathcal{E}$-quasi-weakly Pareto solution,  a type-2 $\mathcal{E}$-quasi-weakly Pareto solution)  defined above coincides with the one of a {\em type-1 Pareto solution} (resp., a {\em type-2 Pareto solution}, a {\em type-1  weakly Pareto solution}, a {\em type-2 weakly Pareto solution}); see, e.g., \cite{Tung-2019b,Osuna-Gomez-17,Wu-09}. 
\begin{remark}\label{Remark-3.1}
	{\rm The following relations are immediate from the definition of approximate optimal solutions.
		\begin{enumerate}[(i)]
			\item $\mathcal{E}$-$\mathcal{S}_1^q\eqref{problem} \subset \mathcal{E}$-$\mathcal{S}_2^q\eqref{problem} \subset \mathcal{E}$-$\mathcal{S}_2^{qw}\eqref{problem}$. 
			
			\item 	$\mathcal{E}$-$\mathcal{S}_1^q\eqref{problem} \subset \mathcal{E}$-$\mathcal{S}_1^{qw}\eqref{problem} \subset \mathcal{E}$-$\mathcal{S}_2^{qw}\eqref{problem}$. 
			
		\end{enumerate}	 
	}
\end{remark}

To obtain the necessary optimality conditions of KKT-type for approximate quasi Pareto solutions of \eqref{problem}, we consider the following constraint qualification condition.
\begin{definition}[see \cite{Chuong-09,Chuong-14}]{\rm 
		Let $\bar x\in\mathcal{F}$. We say that $\bar x$ satisfies the {\em limiting constraint qualification} if the following condition holds
		\begin{equation}\label{LCQ}\tag{LCQ}
		N(\bar x; \mathcal{F})\subseteq \bigcup_{\mu\in A(\bar x)}\left[\sum_{t\in T}\mu_t\partial g_t(\bar x)\right]+N(\bar x; \Omega).
		\end{equation}
	}
\end{definition}

It is worth to mention that the constraint qualification \eqref{LCQ} has been widely used in the literature and it covers almost the existing constraint qualifications of the Mangasarian--Fromovitz and the Farkas--Minkowski types; see e.g., \cite{mor06,Chuong-09,Chuong-14,Chuong-Kim-16,Jiao-Liguo-21}.  

\begin{theorem}\label{Necessary-Theorem} Let $\bar x\in\mathcal{F}$ and assume that $\bar x$ satisfies the \eqref{LCQ}. If $\bar x\in \mathcal{E}$-$\mathcal{S}_2^{qw}\eqref{problem}$, then there exist $\lambda^L, \lambda^U\in \mathbb{R}^m_+$ with $\sum_{i\in I}(\lambda^L_i+\lambda^U_i)=1$, and $\mu\in A(\bar x)$ such that
	\begin{equation}\label{Necessary-condition}
	0\in \sum_{i\in I} \left[\lambda^L_i\partial f^L_i(\bar x)+\lambda^U_i\partial f_i^U(\bar x)\right] +\sum_{t\in T}\mu_t\partial g_t(\bar x)+\sum_{i\in I}\left(\lambda^L_i\epsilon^U_i+\lambda^U_i\epsilon^L_i\right)\mathbb{B}_{\mathbb{R}^n}+N(\bar x; \Omega).
	\end{equation}
\end{theorem}
\begin{proof} Since $\bar x\in \mathcal{E}$-$\mathcal{S}_2^{qw}\eqref{problem}$, there is no $x\in\mathcal{F}$ such that $f_i(x)<^s_{LU} f_i(\bar x)-\mathcal{E}_i$, $\forall i\in I$, or, equivalently, 
	\begin{equation*}
	f_i^L(x)< f_i^L(\bar x)-\epsilon_i^U\|x-\bar x\| \ \ \text{and}\ \ f_i^U(x)< f_i^U(\bar x)-\epsilon_i^L\|x-\bar x\|, \ \ \forall i\in I.  
	\end{equation*}	 
	This means that for each $x\in\mathcal{F}$, there exists $i\in I$ such that	
	\begin{equation}\label{equ-1}
	f_i^L(x)\geq f_i^L(\bar x)-\epsilon_i^U\|x-\bar x\| \ \ \text{or}\ \ f_i^U(x)\geq f_i^U(\bar x)-\epsilon_i^L\|x-\bar x\|.  
	\end{equation}	 
	For each $x\in \mathbb{R}^n$, put 
	$$\varphi(x):=\max_{i\in I}\left\{f_i^L(x)- f_i^L(\bar x)+\epsilon_i^U\|x-\bar x\|, f_i^U(x)-f_i^U(\bar x)+\epsilon_i^L\|x-\bar x\|\right\}.$$
	It follows from \eqref{equ-1} that $\varphi(x)\geq 0$ for all $x\in\mathcal{F}$. Clearly, $\varphi(\bar x)=0$. Hence, $\bar x$ is a minimizer of  $\varphi$ on $\mathcal{F}$, or, equivalently, $\bar x$ is an optimal solution of the following unconstrained optimization problem 
	\begin{equation*}
	\text{minimizer}\ \ \varphi(x)+\delta(x; {\mathcal{F}}), x\in\mathbb{R}^n,
	\end{equation*} 
	where $\delta(\cdot,; {\mathcal{F}})$ is the indicator function of $\mathcal{F}$ and defined by $\delta(x; {\mathcal{F}})=0$ if $x\in\mathcal{F}$ and $\delta(x; {\mathcal{F}})=+\infty$ if $x\notin\mathcal{F}$. By Proposition \ref{Fermat-rule}, we have
	\begin{equation*}
	0\in \partial (\varphi+\delta(\cdot\,; {\mathcal{F}}))(\bar x).
	\end{equation*}
	Since functions $f^L_i$, $f^U_i$, $i\in I$, and $\|\cdot-\bar x\|$ are locally Lipschitz, $\varphi$ is   locally Lipschitz. Clearly, $\delta(\cdot\,; {\mathcal{F}})$ is lower
	semicontinuous around $\bar x$. Hence it follows from Proposition \ref{sum-rule} and the fact that $\partial \delta(\cdot\,; {\mathcal{F}})(\bar x)=N(\bar x; \mathcal{F})$ (see, e.g., \cite[Proposition 1.19]{mor06}) that 
	\begin{equation}\label{equ:2}
	0\in \partial \varphi(\bar x)+N(\bar x; \mathcal{F}).
	\end{equation}   
	By Proposition \ref{max-rule} and the fact that $\partial(\|\cdot-\bar x\|)(\bar x)=\mathbb{B}_{\mathbb{R}^n}$ (see \cite[Example 4, p. 198]{Ioffe79})), we obtain
	\begin{align*}
	\partial \varphi(\bar x) \subset \Big\{\sum_{i\in I}\lambda_i^L [\partial f_i^L(\bar x)+\epsilon_i^U\mathbb{B}_{\mathbb{R}^n}]+\sum_{i\in I}\lambda_i^U [\partial f_i^U(\bar x)+\epsilon_i^L\mathbb{B}_{\mathbb{R}^n}]\;:\; &\lambda_i^L, \lambda_i^U\geq 0, i\in I,
	\\
	&\sum_{i\in I}(\lambda_i^L+\lambda_i^U)=1\Big\}
	\\
	\subset \Big\{\sum_{i\in I}\left[\lambda^L_i\partial f^L_i(\bar x)+\lambda^U_i\partial_i^U(\bar x)\right] +\sum_{i\in I}\left(\lambda^L_i\epsilon^U_i+\lambda^U_i\epsilon^L_i\right)\mathbb{B}_{\mathbb{R}^n}\;:\; &\lambda_i^L, \lambda_i^U\geq 0, i\in I,
	\\
	&\sum_{i\in I}(\lambda_i^L+\lambda_i^U)=1\Big\}.
	\end{align*}
	Combining this, \eqref{equ:2} and the \eqref{LCQ}, we get the assertion.
\end{proof}
The following example is to illustrate Theorem \ref{Necessary-Theorem}.

\begin{example}\label{Example-1}
	{\rm Let  $f:=(f_1, f_2)$, where the functions $f_1, f_2\colon\mathbb{R}^2\to \mathcal{K}_c$ are defined by 
		$$f_1(x)=f_2(x)=[x_1^2+(x_1x_2-1)^2, 2x_1^2+(x_1x_2-1)^2], \ \ \forall x=(x_1, x_2)\in\mathbb{R}^2,$$
		and let $g_t\colon\mathbb{R}^2\to \mathbb{R}$ be given by
		$$g_t(x)=-t|x_1|-t|x_2|, \ \ \forall x=(x_1, x_2)\in\mathbb{R}^2, t\in T:=[0,1].$$
		Consider problem \eqref{problem} with $\Omega=\mathbb{R}^2$. Clearly, the feasible set of \eqref{problem} is  $\mathcal{F}=\mathbb{R}^2$.  
		
		Let $\mathcal{E}=(\mathcal{E}_1, \mathcal{E}_2)$, where $\mathcal{E}_1=\mathcal{E}_2=[1, 2]$. We see that $\bar x=(0, 0)\in \mathcal{E}$-$\mathcal{S}_1^q\eqref{problem}$ and so by Remark \ref{Remark-3.1},  $\bar x\in \mathcal{E}$-$\mathcal{S}_2^{qw}\eqref{problem}$. Indeed, suppose on the contrary that $\bar x\notin \mathcal{E}$-$\mathcal{S}_1^q\eqref{problem}$,then there exists $x\in\mathbb{R}^2$ such that
		$$f_1(x)<_{LU} f_1(\bar x)-\mathcal{E}_1\|x-\bar x\|,$$
		or, equivalently, 
		\begin{equation}\label{equa:1}
		\begin{cases}
		x_1^2+(x_1x_2-1)^2\leq 1-2\sqrt{x_1^2+x_2^2},
		\\
		2x_1^2+(x_1x_2-1)^2\leq 1-\sqrt{x_1^2+x_2^2},
		\end{cases}
		\end{equation}
		with at least one strict inequality. Since \eqref{equa:1}, we have
		\begin{equation*} 
		\begin{cases}
		x_1^2+x_1^2x^2_2+ 2 \sqrt{x_1^2+x_2^2}-2x_1x_2\leq 0
		\\
		2x_1^2+x_1^2x^2_2+  \sqrt{x_1^2+x_2^2}-2x_1x_2\leq 0
		\end{cases}
		\Leftrightarrow
		\begin{cases}
		x_1=0
		\\
		x_2=0.
		\end{cases}
		\end{equation*}
		Hence, there is no strict inequality in \eqref{equa:1}, a contradiction.
		
		Since $N(x; \mathcal{F})=N(x; \Omega)=\partial g_0(x)=\{0\}$, the \eqref{LCQ} is satisfied at every $x\in\mathbb{R}^2$. Hence, by Theorem \ref{Necessary-Theorem}, there exist $\lambda^L, \lambda^U\in \mathbb{R}^2_+$ with $\sum_{i\in I}(\lambda^L_i+\lambda^U_i)=1$, and $\mu\in A(\bar x)$ satisfying condition \eqref{Necessary-condition}.
		
		We note here that the exact solution set of the problem is empty. Indeed, let $x^*$ be an arbitrary point in $\mathbb{R}^2$. Then, $0<f_i^L(x^*)\leq f_i^U(x^*)$ for all $i\in I$. Let $x^k=(\frac{1}{k}, k)$, $k\in\mathbb{N}$.  Then, for all $i\in I$, we have
		\begin{align*}
		&\lim\limits_{k\to\infty}f_i^L(x^k)=\lim\limits_{k\to\infty}\frac{1}{k^2}=0<f_i^L(x^*),
		\\
		&\lim\limits_{k\to\infty}f_i^U(x^k)=\lim\limits_{k\to\infty}\frac{2}{k^2}=0<f_i^U(x^*).
		\end{align*}
		Hence, there exists $k_0\in\mathbb{N}$ such that
		\begin{align*}
		f_i^L(x^k)< f_i^L(x^*),
		\\
		f_i^U(x^k)< f_i^U(x^*),
		\end{align*} 
		for all $k\geq k_0$ and $i\in I$. This means that $x^*$ is not a type-2 weakly Pareto solution of \eqref{problem}, as required.
		\begin{center}
			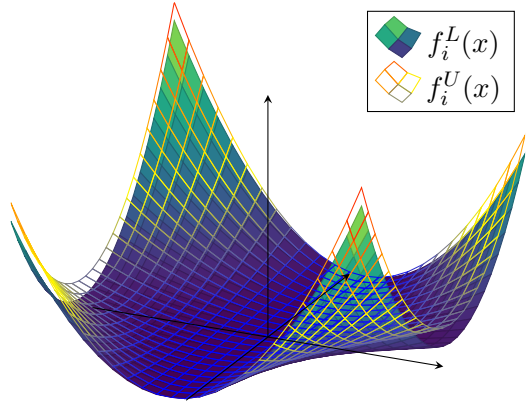
\begin{figure}[htp]
				\begin{center}
					\begin{tikzpicture}[]
					\begin{axis}[axis lines=center,
					axis on top,
					xtick=\empty,
					ytick=\empty,
					ztick=\empty,
					xrange=-3:3,
					yrange=-3:3
					]
					\addplot3[domain=-3:3,y domain=-3:3,colormap/viridis,surf]
					{x^2+(x*y-1)^2};
					\addlegendentry{$f^L_{i}(x)$}
					\addplot3[domain=-3:3,y domain=-3:3,red,mesh]
					{2*x^2+(x*y-1)^2};
					\addlegendentry{$f^U_{i}(x)$}
					\end{axis}
					\end{tikzpicture}
					\caption{Plot of $f^L_i$ and $f^U_i$ in Example \ref{Example-1}.}\label{Fig.2}
				\end{center}
			\end{figure}
		\end{center}
	}
\end{example}

Next we present sufficient conditions for approximate quasi Pareto solutions of \eqref{problem}.  In order to obtain these sufficient conditions, we need to introduce concepts of (strictly) generalized convexity  at a given point for a family of locally Lipschitz functions. The first definition is inspired from  \cite{Chuong-14}, while the second one is motivated from  \cite{Fakhar-19}. 

\begin{definition}
	{\rm 
		\begin{enumerate}[(i)]
			\item  We say that $(f, g_T)$ is {\em generalized  convex on $\Omega$} at $\bar x\in\Omega$ if for any $x\in\Omega$, $z_i^{*L}\in\partial f_i^L(\bar x)$,  $z_i^{*U}\in\partial f_i^U(\bar x)$, $i\in I$, and $x^*_t\in\partial g_t(\bar x)$, $t\in T$, there exists $\nu\in [N(\bar x; \Omega)]^\circ$ satisfying
			\begin{equation}\label{equa:9} 
			\begin{split}
			&f_i^L(x) -f_i^L(\bar x)\geq \langle z_i^{*L}, \nu\rangle, \ \ \forall  i\in I,
			\\
			&f_i^U(x) -f_i^U(\bar x)\geq \langle z_i^{*U}, \nu\rangle, \ \ \forall  i\in I,
			\\
			&g_t(x)- g_t(\bar x) \geq \langle x^*_t, \nu\rangle,\ \ \forall t\in T\\
			&\text{and}\ \ \langle b^*,\nu\rangle\leq \|x-\bar x\|,\ \ \forall b^*\in\mathbb{B}_{\mathbb{R}^n}.
			\end{split}
			\end{equation}
			\item  	  
			We say that $(f, g_T)$ is {\em strictly generalized  convex on $\Omega$} at $\bar x\in\Omega$ if for any $x\in\Omega\setminus\{\bar x\}$, $z_i^{*L}\in\partial f_i^L(\bar x)$,  $z_i^{*U}\in\partial f_i^U(\bar x)$, $i\in I$, and $x^*_t\in\partial g_t(\bar x)$, $t\in T$, there exists $\nu\in [N(\bar x; \Omega)]^\circ$ satisfying
			\begin{equation*} 
			\begin{split}
			&f_i^L(x) -f_i^L(\bar x)> \langle z_i^{*L}, \nu\rangle, \ \ \forall  i\in I,
			\\
			&f_i^U(x) -f_i^U(\bar x)> \langle z_i^{*U}, \nu\rangle, \ \ \forall  i\in I,
			\\
			&g_t(x)- g_t(\bar x) \geq \langle x^*_t, \nu\rangle,\ \ \forall t\in T\\
			&\text{and}\ \ \langle b^*,\nu\rangle\leq \|x-\bar x\|,\ \ \forall b^*\in\mathbb{B}_{\mathbb{R}^n}.
			\end{split}
			\end{equation*}
		\end{enumerate}
	}
\end{definition}
\begin{remark}
	{\rm 
		We see that if $\Omega$ is convex and $f^L_i$, $f^U_i$, $i\in I$, and $g_t$, $t\in T$, are convex (resp. strictly convex), then $(f, g_T)$ is generalized convex (resp. strictly generalized convex) on $\Omega$ at any $\bar x\in \Omega$ with $\nu=x-\bar x$. Moreover, there exist examples that show the class of generalized convex functions is properly larger than the one of convex functions; see, e.g., \cite[Example 3.2]{Chuong-14} and \cite[Example 3.12]{Chuong-Kim-16}.  
	}
\end{remark}
\begin{definition}
	{\rm  
		\begin{enumerate}[(i)]
			\item   We say that $(f, g_T)$ is {\em $\mathcal{E}$-pseudo-quasi generalized  convex on $\Omega$} at $\bar x\in\Omega$ if for any $x\in\Omega$, $z_i^{*L}\in\partial f_i^L(\bar x)$,  $z_i^{*U}\in\partial f_i^U(\bar x)$, $i\in I$, and $x^*_t\in\partial g_t(\bar x)$, $t\in T$, there exists $\nu\in [N(\bar x; \Omega)]^\circ$ satisfying
			\begin{equation}\label{equ:3}
			\begin{split}
			&\langle z_i^{*L}, \nu\rangle +\epsilon_i^U\|x-\bar x\|\geq 0 \Rightarrow f_i^L(x) \geq f_i^L(\bar x)-\epsilon_i^U\|x-\bar x\|, \ \ \forall  i\in I,
			\\
			&\langle z_i^{*U}, \nu\rangle +\epsilon_i^L\|x-\bar x\|\geq 0 \Rightarrow f_i^U(x) \geq f_i^U(\bar x)-\epsilon_i^L\|x-\bar x\|, \ \  \forall i\in I,
			\\
			&g_t(x)\leq g_t(\bar x) \Rightarrow \langle x^*_t, \nu\rangle\leq 0,\ \ \forall t\in T\\
			&\text{and}\ \ \langle b^*,\nu\rangle\leq \|x-\bar x\|,\ \ \forall b^*\in\mathbb{B}_{\mathbb{R}^n}.
			\end{split}
			\end{equation}
			\item  	  
			We say that $(f, g_T)$ is {\em strictly $\mathcal{E}$-pseudo-quasi generalized  convex on $\Omega$} at $\bar x\in\Omega$ if for any $x\in\Omega\setminus\{\bar x\}$, $z_i^{*L}\in\partial f_i^L(\bar x)$,  $z_i^{*U}\in\partial f_i^U(\bar x)$, $i\in I$, and $x^*_t\in\partial g_t(\bar x)$, $t\in T$, there exists $\nu\in [N(\bar x; \Omega)]^\circ$ satisfying
			\begin{equation}\label{equa:8}
			\begin{split}
			&\langle z_i^{*L}, \nu\rangle +\epsilon_i^U\|x-\bar x\|\geq 0 \Rightarrow f_i^L(x) > f_i^L(\bar x)-\epsilon_i^U\|x-\bar x\|, \ \ \forall  i\in I, 
			\\
			&\langle z_i^{*U}, \nu\rangle +\epsilon_i^L\|x-\bar x\|\geq 0 \Rightarrow f_i^U(x) > f_i^U(\bar x)-\epsilon_i^L\|x-\bar x\|,\ \  \forall i\in I, 
			\\
			&g_t(x)\leq g_t(\bar x) \Rightarrow \langle x^*_t, \nu\rangle\leq 0, \ \ \forall t\in T, 
			\\
			&\text{and}\ \ \langle b^*,\nu\rangle\leq \|x-\bar x\|,\ \ \forall b^*\in\mathbb{B}_{\mathbb{R}^n}. 
			\end{split}
			\end{equation}
		\end{enumerate}
	}
\end{definition}

\begin{remark}\label{Remark-3}
	{\rm By definition, it is easy to see that if  $(f, g_T)$ is  (strictly) generalized  convex on $\Omega$ at $\bar x\in\Omega$, then for any $\mathcal{E}=(\mathcal{E}_1, \ldots, \mathcal{E}_m)$, where $\mathcal{E}_i=[\epsilon_i^L, \epsilon_i^u]$, $0\leq\epsilon_i^L\leq \epsilon_i^U$, $i\in I$,   $(f, g_T)$ is (strictly) $\mathcal{E}$-pseudo-quasi generalized  convex on $\Omega$ at $\bar x\in\Omega$. Furthermore, the class of  (strictly) $\mathcal{E}$-pseudo-quasi generalized  convex functions is  properly wider than the one of (strictly) generalized convex functions. To see this, let us consider the following example.
	}
\end{remark}

\begin{example}\label{Example-2}
	{\rm  Let $m=1$, $f\colon \mathbb{R}\to\mathcal{K}_c$ be defined by $f(x)=[f^L(x), f^U(x)]$, where
		\begin{equation*}
		f^L(x)=
		\begin{cases}
		\dfrac{x}{2},  & x\geq 0,
		\smallskip
		\\
		\dfrac{2x}{3},  & x<0,
		\end{cases}
		\ \ \ \ 
		f^U(x)=
		\begin{cases}
		\dfrac{2x}{3},  & x\geq 0,
		\smallskip
		\\
		\dfrac{x}{2},  & x<0,
		\end{cases}
		\end{equation*}	
		and $g_t(x)=-tx$ for all $x\in\mathbb{R}$ and $t\in T:=[0, 1]$.	 Take $\Omega=\mathbb{R}$ and $\mathcal{E}=[\frac{2}{3}, \frac{3}{4}]$. We show that $(f, g_T)$ is strictly $\mathcal{E}$-pseudo-quasi generalized  convex at $\bar x=0$ but not  generalized  convex at this point. Indeed, it is easy to see that
		\begin{equation*}
		f^L(x)>f^L(\bar x)-\frac{3}{4}|x|, f^U(x)>f^U(\bar x)-\frac{2}{3}|x|, \ \ \forall x\in\mathbb{R}\setminus\{0\}.
		\end{equation*}
		Hence, by letting $\nu=0$, we see that all conditions in \eqref{equa:8} are satisfied for all  $z^{*L}\in\partial f^L(\bar x)$,  $z^{*U}\in\partial f^U(\bar x)$, $x^*_t\in\partial g_t(\bar x)$, $t\in T$, and $b^*\in\mathbb{B}_{\mathbb{R}^n}$.
		
		We now show that $(f, g_T)$ is  not generalized  convex at $\bar x$ and so it is not strictly generalized  convex at this point. We have $\partial f^L(\bar x)=\partial f^U(\bar x)=\{\frac{1}{2}, \frac{2}{3}\}$ and $\partial g_t(\bar x)=\{-t\}$, $\forall t\in T$. Let $x=1$, $z^{*L}=z^{*U}=\frac{2}{3}$, $x^{*}_t=-t$, and assume that there exists $\nu\in\mathbb{R}$ satisfying  \eqref{equa:9}. Then
		\begin{align*}
		\frac{1}{2}&\geq \frac{2\nu}{3},
		\\
		\frac{2}{3}&\geq \frac{2\nu}{3},
		\\
		-t&\geq -t\nu, \ \ \forall t\in [0, 1].
		\end{align*}    
		Hence, $1 \leq \nu\leq \frac{3}{4},$ a contradiction.   
		\begin{center}
			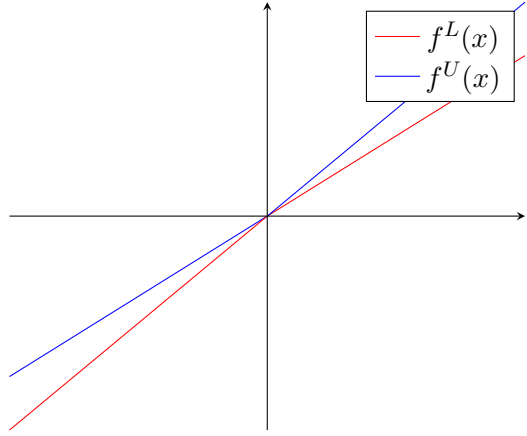
\begin{figure}[htp]
				\begin{center}
					\begin{tikzpicture}[]
					\begin{axis}[axis lines=center,
					axis on top,
					xtick=\empty,
					ytick=\empty,
					xrange=-5:5,
					yrange=-5:5
					]
					\addplot[domain=0:5,y domain=-5:5,red]
					{x/2};
					\addlegendentry{$f^L(x)$}	
					\addplot[domain=-5:0,y domain=-5:5,blue]
					{x/2};
					\addlegendentry{$f^U(x)$}		
					\addplot[domain=-5:0,y domain=-5:5,red]
					{2*x/3};
					\addplot[domain=0:5,y domain=-5:5,blue]
					{2*x/3};			
					\end{axis}
					\end{tikzpicture}
					\caption{Plot of $f^L$ and $f^U$ in Example \ref{Example-2}.}\label{Fig.1}
				\end{center}
			\end{figure}
		\end{center}
	}
\end{example}

\begin{theorem}\label{Sufficient-Theorem} Let $\bar x\in\mathcal{F}$ and assume that there exist $\lambda^L, \lambda^U\in \mathbb{R}^m_+$ with $\sum_{i\in I}(\lambda_i^L+\lambda_i^U)=1$, and $\mu\in A(\bar x)$ satisfying \eqref{Necessary-condition}. 
	\begin{enumerate}[\rm(i)]
		\item If $(f, g_T)$ is $\mathcal{E}$-pseudo-quasi generalized convex on $\Omega$ at $\bar x$, then $\bar x\in \mathcal{E}$-$\mathcal{S}_2^{qw}\eqref{problem}$.
		
		\item If $(f, g_T)$ is strictly $\mathcal{E}$-pseudo-quasi generalized convex on $\Omega$ at $\bar x$, then $\bar x\in \mathcal{E}$-$\mathcal{S}_1^{q}\eqref{problem}$ and we therefore get $\bar x\in \mathcal{E}$-$\mathcal{S}_2^{q}\eqref{problem}$ and $\bar x\in \mathcal{E}$-$\mathcal{S}_1^{qw}\eqref{problem}$.  
	\end{enumerate}
\end{theorem}	
\begin{proof} Since $\bar x$ satisfies \eqref{Necessary-condition} with respect to $(\lambda^L, \Lambda^U, \mu)$, there exist $z_i^{*L}\in \partial f_i^L(\bar x)$, $z_i^{*U}\in \partial f_i^U(\bar x)$, $i\in I$, $x_t^*\in \partial g_t(\bar x)$, $t\in T$, $b^*\in \mathbb{B}_{\mathbb{R}^n}$, and $\omega^*\in N(\bar x; \Omega)$ such that
	\begin{equation*}
	\sum_{i\in I}[\lambda_i^L z_i^{*L}+ \lambda_i^U z_i^{*U}]+\sum_{t\in T}\mu_t x^*_t+\sum_{i\in I}(\lambda_i^L\epsilon_i^U+\lambda_i^U\epsilon_i^L)b^*+\omega^*=0,
	\end{equation*}	
	or, equivalently,
	\begin{equation}\label{equ:5}
	\sum_{i\in I}[\lambda_i^L z_i^{*L}+ \lambda_i^U z_i^{*U}]+\sum_{t\in T}\mu_t x^*_t+\sum_{i\in I}(\lambda_i^L\epsilon_i^U+\lambda_i^U\epsilon_i^L)b^*=-\omega^*.
	\end{equation}
	We first justify (i). Suppose on the contrary that $\bar x$ is not a type-2 $\mathcal{E}$-quasi-weakly Pareto solution of \eqref{problem}, then there exists $x\in\mathcal{F}$ such that $$f_i(x)<^s_{LU} f_i(\bar x)-\mathcal{E}_i\|x-\bar x\|,\ \  \forall i\in I,$$
	or, equivalently, 
	\begin{equation}\label{equ:4}
	f_i^L(x)< f_i^L(\bar x)-\epsilon_i^U\|x-\bar x\| \ \ \text{and}\ \ f_i^U(x)< f_i^U(\bar x)-\epsilon_i^L\|x-\bar x\|, \ \ \forall i\in I.  
	\end{equation}	 
	By the $\mathcal{E}$-pseudo-quasi generalized convexity of $(f, g_T)$, there is $\nu\in [N(\bar x; \Omega)]^\circ$ satisfying \eqref{equ:3}. Thus, we deduce from \eqref{equ:3} and \eqref{equ:4} that
	\begin{equation*}
	\langle z_i^{*L}, \nu\rangle +\epsilon_i^U\|x-\bar x\|< 0\ \ \text{and}\ \ \langle z_i^{*U}, \nu\rangle +\epsilon_i^L\|x-\bar x\|< 0.
	\end{equation*}   
	For each $t\in T(\mu)$, we have $g_t(\bar x)=0$. Hence, 
	$g_t(x)\leq g_t(\bar x)$ for all $t\in T(\mu)$. This, together with \eqref{equ:3}, gives 
	\begin{equation*}
	\langle x^*_t, \nu\rangle\leq 0\ \ \text{for all}\ \ t\in T(\mu).
	\end{equation*}
	Since  $\nu\in [N(\bar x; \Omega)]^\circ$, \eqref{equ:3}, and \eqref{equ:5}, we obtain
	\begin{align*}
	0&\leq \langle -\omega^*,\nu\rangle=\sum_{i\in I}[\lambda_i^L \langle z_i^{*L}, \nu\rangle+ \lambda_i^U \langle z_i^{*U}, \nu\rangle]+\sum_{t\in T}\mu_t\langle x^*_t, \nu\rangle+\sum_{i\in I}(\lambda_i^L\epsilon_i^U+\lambda_i^U\epsilon_i^L)\langle b^*,\nu\rangle 
	\\
	&=\sum_{i\in I}[\lambda_i^L \langle z_i^{*L}, \nu\rangle+ \lambda_i^U \langle z_i^{*U}, \nu\rangle]+\sum_{t\in T(\mu)}\mu_t\langle x^*_t, \nu\rangle+\sum_{i\in I}(\lambda_i^L\epsilon_i^U+\lambda_i^U\epsilon_i^L)\langle b^*,\nu\rangle
	\\
	&<-\sum_{i\in I}[\lambda_i^L\epsilon_i^U\|x-\bar x\|+\lambda_i^U\epsilon_i^L\|x-\bar x\|]+\sum_{i\in I}(\lambda_i^L\epsilon_i^U+\lambda_i^U\epsilon_i^L)\|x-\bar x\|=0,
	\end{align*} 
	a contradiction. The proof of (i) is completed.  
	
	We now prove (ii). Suppose on the contrary that $\bar x$ is not a type-1 $\mathcal{E}$-quasi  Pareto solution of \eqref{problem}, then there exists $x\in\mathcal{F}$ such that 
	$$f_i(x)\leq_{LU} f_i(\bar x)-\mathcal{E}_i\|x-\bar x\|,\ \ \forall i\in I,$$   	 
	where at least one of the inequalities is strict. This implies that $x\neq \bar x$. Therefore, by the strictly $\mathcal{E}$-pseudo-quasi generalized convexity of $(f, g_T)$ at $\bar x$, \eqref{equ:5}, and by the same argument as in the proof of  part (i), we can deduce the contradiction. Hence, $\bar x$ is  a type-1 $\mathcal{E}$-quasi  Pareto solution of \eqref{problem}.  The proof is completed. 
\end{proof} 

\begin{remark}{\rm 
		\begin{enumerate}[(i)]
			\item By Remark \ref{Remark-3} and \cite[Example 3.2]{Son-Tuyen-Wen-19},  the condition \eqref{Necessary-Theorem} alone is not sufficient to 	guarantee that $\bar x$ is a $\mathcal{E}$-quasi (-weakly) Pareto solution of \eqref{problem} if the (strict) $\mathcal{E}$-pseudo generalized convexity of 	$(f, g_T)$ on $\Omega$ at $\bar x$ is violated.
			
			\item If $(f, g_T)$ is $\mathcal{E}$-pseudo-quasi generalized convex on $\Omega$ at $\bar x\in \mathcal{F}$ and there exist $\lambda^L, \lambda^U\in\mathbb{R}^m_+$ with $\lambda_i^L>0, \lambda_i^U>0$, $\forall i\in I$, $\sum_{i\in I}(\lambda_i^L+\lambda_i^U)=1$, and $\mu\in A(\bar x)$ satisfying \eqref{Necessary-condition}, then $\bar x\in \mathcal{E}$-$\mathcal{S}_1^{q}\eqref{problem}$.
			
			\item  Since the class of (strictly) $\mathcal{E}$-pseudo-quasi generalized convex functions is properly wider than the class of (strictly) generalized convex functions, our results in Theorem \ref{Necessary-Theorem} generalize and improve the corresponding results in \cite{Chuong-Kim-16,Jiao-Liguo-21,Son-Tuyen-Wen-19,Tuyen-2021}.  To see this, let us consider the following simple example.
		\end{enumerate}
	}
\end{remark}

\begin{example} {\rm 
		Let $m=1$, $f\colon \mathbb{R}\to\mathcal{K}_c$ be defined by $f(x)=[f^L(x), f^U(x)]$, where
		\begin{equation*}
		f^L(x)=f^U(x)=
		\begin{cases}
		\dfrac{x}{2},  & x\geq 0,
		\smallskip
		\\
		\dfrac{2x}{3},  & x<0.
		\end{cases}
		\end{equation*}	
		Consider problem \eqref{problem} with $g_t(x)=-tx$ for all $x\in\mathbb{R}$ and $t\in T:=[0, 1]$, $\Omega=\mathbb{R}$ and $\mathcal{E}=[\frac{2}{3}, \frac{2}{3}]$. Actually, in this case problem \eqref{problem} is a  semi-infinite programming problem.  Analysis similar to that in Example  \ref{Example-2} shows that $(f, g_T)$ is strictly $\mathcal{E}$-pseudo-quasi convex at $\bar x=0$ but not generalized convex at this point. It is easy to see that $\bar x$ satisfies condition \eqref{Necessary-condition}; e.g., $\lambda^L=\lambda^U=\frac{1}{2}$, $\mu_{\frac{1}{2}}=1$,  and $\mu_t=0$ for all $t\in T\setminus\{\frac{1}{2}\}$. Hence, by Theorem \ref{Sufficient-Theorem}, $\bar x\in \mathcal{E}$-$\mathcal{S}_1^{q}\eqref{problem}$. However, since  $(f, g_T)$ is not  generalized convex at $\bar x$, \cite[Theorem 3.13]{Chuong-Kim-16}, \cite[Theorem 3.2]{Jiao-Liguo-21}, \cite[Theorem 3.3]{Son-Tuyen-Wen-19}, and \cite[Theorem 6]{Tuyen-2021} cannot be applied for this example.   
	} 
\end{example} 
\section{Approximate duality theorems}\label{Duality-Relations}
Let $\mathcal{A}:=(A_1, \ldots, A_m)$ and $\mathcal{B}:=(B_1, \ldots, B_m)$, where $A_i$, $B_i$, $i\in I$, are intervals in $\mathcal{K}_c$. In what follows,   we use the following notations for convenience.
\begin{align*}
\mathcal{A}&\preceq_{LU}\mathcal{B} \Leftrightarrow
\begin{cases}
A_i\leq_{LU} B_i, \ \ \forall i\in I,
\\
A_k<_{LU} B_k,  \ \ \text{for at least one}\ \ k\in I. 
\end{cases}
\\
\mathcal{A}&\npreceq_{LU}\mathcal{B} \ \ \text{is the negation of}\ \ \mathcal{A}\preceq_{LU}.
\\
\mathcal{A}&\prec^s_{LU}\mathcal{B}  \Leftrightarrow A_i<^s_{LU} B_i, \ \ \forall i\in I.
\\
\mathcal{A}&\nprec^s_{LU}\mathcal{B} \ \ \text{is the negation of}\ \ \mathcal{A}\prec^s_{LU}\mathcal{B}. 
\end{align*}

For $y\in\mathbb{R}^n$, $(\lambda^L, \lambda^U)\in\mathbb{R}^m_+\times\mathbb{R}^m_+\setminus\{(0,0)\}$, and $\mu\in \mathbb{R}^{|T|}_+$, put
\begin{align*}
\mathcal{L}(y, \lambda^L,\lambda^U, \mu):=f(y)=\left([f_1^L(y), f_1^U(y)], \ldots, [f_m^L(y), f_m^U(y)]\right).
\end{align*} 

In connection with the primal problem \eqref{problem}, we consider the following dual problem in the sense of Mond--Weir (stated in an approximate form):
\begin{align}
\label{Dual-problem}\tag{SIVD$_{MW}$} 
&\max_{y\in\Omega}\ \  \mathcal{L}(y, \lambda^L,\lambda^U, \mu) 
\\
&\text{s. t.}\ \ (y, \lambda^L,\lambda^U, \mu)\in\mathcal{F}_{MW},\notag
\end{align}
where the feasible set is defined by
\begin{align*}
&\mathcal{F}_{MW}:=\big\{(y, \lambda^L, \lambda^U, \mu)\in \Omega\times\mathbb{R}^m_+ \times\mathbb{R}^m_+\times\mathbb{R}^{|T|}_+\,:\, 0\in \sum_{i\in I}[\lambda_i^L\partial f_i^L(y)+\lambda_i^U\partial f_i^U(y)]+
\\
&\sum_{t\in T}\mu_t\partial g_t(y)+\sum_{i\in I}(\lambda_i^L\epsilon_i^U+\lambda_i^U\epsilon_i^L)\mathbb{B}_{\mathbb{R}^n}+N(y; \Omega), \mu_tg_t(y)\geq 0, t\in T,  \sum_{i\in I}(\lambda_i^L+\lambda_i^U)=1\big\}.
\end{align*} 
It should be noticed that approximate  quasi Pareto solutions of the dual problem \eqref{Dual-problem}  are defined similarly as in Definition \ref{Defi-solution} by replacing $\leq_{LU}$, $<_{LU}$, and $<^s_{LU}$ by  $\geq_{LU}$, $>_{LU}$, and $>^s_{LU}$, respectively. For example, we say that $(\bar y, \bar  \lambda^L,\bar \lambda^U, \bar \mu)\in\mathcal{F}_{MW}$ is a {\em type-2 $\mathcal{E}$-quasi weakly Pareto solution} of \eqref{Dual-problem} if there is no $(y, \lambda^L, \lambda^U, \mu)\in \mathcal{F}_{MW}$ such that    
\begin{equation*}
\mathcal{L}_i(y, \lambda^L,\lambda^U, \mu)-\mathcal{E}_i\|y-\bar y\|>^s_{LU} \mathcal{L}_i(\bar y, \bar \lambda^L,\bar\lambda^U, \bar\mu),\ \ \forall i\in I,
\end{equation*}
or, equivalently,
\begin{equation*}
\mathcal{L}(\bar y, \bar \lambda^L,\bar\lambda^U, \bar\mu)\nprec^s_{LU}\mathcal{L}_i(y, \lambda^L,\lambda^U, \mu)-\mathcal{E}_i\|y-\bar y\|,\ \ \forall (y, \lambda^L, \lambda^U, \mu)\in \mathcal{F}_{MW}.
\end{equation*}

The following theorem describes weak duality relations for approximate quasi Pareto solutions between the primal problem \eqref{problem} and the dual problem \eqref{Dual-problem}.
\begin{theorem}[$\mathcal{E}$-weak duality]\label{weak-dual} Let $x\in \mathcal{F}$ and $(y, \lambda^L, \lambda^U, \mu)\in\mathcal{F}_{MW}$.
	\begin{enumerate}[\rm(i)]
		\item If $(f, g_T)$ is $\mathcal{E}$-pseudo-quasi generalized convex on $\Omega$ at $y$, then
		\begin{equation*}
		f(x) \nprec^s_{LU} \mathcal{L}(y, \lambda^L, \lambda^U, \mu) -\mathcal{E}\|x-y\|.
		\end{equation*}
		\item If $(f, g_T)$ is strictly $\mathcal{E}$-pseudo-quasi generalized convex on $\Omega$ at $y$, then
		\begin{equation*}
		f(x)\npreceq_{LU} \mathcal{L}(y, \lambda^L, \lambda^U, \mu)-\mathcal{E}\|x-y\|.
		\end{equation*}
	\end{enumerate}
\end{theorem}
\begin{proof} Since $x\in\mathcal{F}$ and $(y, \lambda^L, \lambda^U, \mu)\in\mathcal{F}_{MW}$,  we have $x, y\in\Omega$,  
	\begin{equation}\label{equ:7}
	g_t(x)\leq 0, \ \ \mu_tg_t(y)\geq 0, \ \ \forall t\in T, 
	\end{equation}
	and 
	\begin{equation*}
	0\in \sum_{i\in I}[\lambda_i^L\partial f_i^L(y)+\lambda_i^U\partial f_i^U(y)]+\sum_{t\in T}\mu_t\partial g_t(y)+\sum_{i\in I}(\lambda_i^L\epsilon_i^U+\lambda_i^U\epsilon_i^L)\mathbb{B}_{\mathbb{R}^n}+N(y; \Omega).
	\end{equation*}
	Hence, there exist $z_i^{*L}\in \partial f_i^L(y)$, $z_i^{*U}\in \partial f_i^U(y)$, $i\in I$, $x^*_t\in\partial g_t(y)$, $t\in T$, $b^*\in\mathbb{B}_{\mathbb{R}^n}$, and $\omega^*\in N(y; \Omega)$ such that
	\begin{equation}\label{equ:6}
	\sum_{i\in I}[\lambda_i^L z_i^{*L}+ \lambda_i^U z_i^{*U}]+\sum_{t\in T}\mu_t x^*_t+\sum_{i\in I}(\lambda_i^L\epsilon_i^U+\lambda_i^U\epsilon_i^L)b^*=-\omega^*.
	\end{equation}	
	
	We first justify (i). Assume to the contrary that
	\begin{equation*}
	f(x)\prec^s_{LU} \mathcal{L}(y, \lambda^L, \lambda^U, \mu)-\mathcal{E}\|x-y\|.
	\end{equation*}
	This means that
	\begin{equation*}
	f_i(x)\prec^s_{LU} \mathcal{L}_i(y, \lambda^L, \lambda^U, \mu) -\mathcal{E}_i\|x-y\|, \ \ \forall i\in I,
	\end{equation*}
	or, equivalently,
	\begin{align*}
	\begin{cases}
	f_i^L(x)<f_i^L(y)-\epsilon_i^U\|x-y\|,
	\\
	f_i^U(x)<f_i^U(y)-\epsilon_i^L\|x-y\|, 
	\end{cases}
	\end{align*}
	for all $i\in I$. By \eqref{equ:7} and $\mu_t\geq 0$ for all $t\in T$, we have
	\begin{equation*}
	g_t(x)\leq 0\leq g_t(y)
	\end{equation*}
	for all $t\in T(\mu)$. Since $(f, g_T)$ is  $\mathcal{E}$-pseudo-quasi generalized convex  on $\Omega$ at $y$, there exists $\nu\in [N(y, \Omega)]^\circ$ such that
	\begin{align*}
	&\langle z_i^{*L}, \nu\rangle +\epsilon_i^U\|x-y\|<0, \ \ \forall i\in I,
	\\
	&\langle z_i^{*U}, \nu\rangle +\epsilon_i^L\|x-y\|<0, \ \ \forall i\in I,
	\\
	&\langle x^*_t, \nu\rangle \leq 0, \ \  \forall t\in T(\mu), 
	\\
	&\text{and}\ \ \langle b^*,\nu\rangle\leq \|x-y\|,\ \ \forall b^*\in\mathbb{B}_{\mathbb{R}^n}.
	\end{align*}
	Since $\sum_{i\in I}(\lambda_i^L+\lambda_i^U)=1$, one has
	\begin{align*}
	\sum_{i\in I}\big[\langle \lambda_i^Lz_i^{*L}&+\lambda_i^Uz_i^{*U}, \nu\rangle\big]+\sum_{t\in T}\langle x^*_t, \nu\rangle+\sum_{i\in I}(\lambda_i^L\epsilon_i^U+\lambda_i^U\epsilon_i^L)\langle b^*, \nu\rangle \leq 
	\\
	&\sum_{i\in I}\big[\langle \lambda_i^Lz_i^{*L}+\lambda_i^Uz_i^{*U}, \nu\rangle\big]+\sum_{t\in T}\langle x^*_t, \nu\rangle+\sum_{i\in I}(\lambda_i^L\epsilon_i^U+\lambda_i^U\epsilon_i^L)\|x-y\|<0.
	\end{align*}
	On the other hand, since  $\nu\in [N(y, \Omega)]^\circ$, $\omega^*\in N(y; \Omega)$, and \eqref{equ:6}, we have 
	\begin{equation*}
	0\leq \langle -\omega^*, \nu\rangle= \sum_{i\in I}\big[\langle \lambda_i^Lz_i^{*L}+\lambda_i^Uz_i^{*U}, \nu\rangle\big]+\sum_{t\in T}\langle x^*_t, \nu\rangle+\sum_{i\in I}(\lambda_i^L\epsilon_i^U+\lambda_i^U\epsilon_i^L)\langle b^*, \nu\rangle <0,
	\end{equation*}
	a contradiction, which completes the proof of (i).
	
	We now prove (ii). Suppose on the contrary that 
	\begin{equation*}
	f(x)\preceq_{LU} \mathcal{L}(y, \lambda^L, \lambda^U, \mu)-\mathcal{E}\|x-y\|.
	\end{equation*} 
	This implies that 
	\begin{equation*}
	\begin{cases}
	f_i(x)\leq_{LU} f_i(y)-\mathcal{E}_i\|x-y\|, \ \ \forall i\in I,
	\\
	f_k(x)<_{LU} f_k(y)-\mathcal{E}_k\|x-y\|, \ \ \text{for at least one} \ \ k\in I.
	\end{cases}
	\end{equation*}
	Hence, $x\neq y$.  Therefore, by the strictly $\mathcal{E}$-pseudo-quasi generalized convexity of $(f, g_T)$ at $y$, \eqref{equ:6}, and by the same argument as in the proof of  part (i), we can get the contradiction. The proof is completed.  
\end{proof}
The following example shows that the approximate pseudo-quasi generalized convexity of $(f,g_T)$ on $\Omega$ used in Theorem \ref{weak-dual} cannot be omitted.

\begin{example}\label{Example-3}
	{\rm  Let $m=1$, $f\colon \mathbb{R}\to\mathcal{K}_c$ be defined by $f(x)=[f^L(x), f^U(x)]$, where 
		$$f^L(x)=\frac{1}{3}x^3 \ \ \text{and}\ \ f^U(x)=\frac{1}{3}x^3+1.$$
		Consider problem \eqref{problem} with $g_t(x)=-tx$ for all $x\in\mathbb{R}$ and $t\in T:=[0, 1]$, $\Omega=\mathbb{R}$ and $\mathcal{E}=[\frac{1}{5}, \frac{1}{4}]$.  Let $\bar y=1$, $\bar \lambda^L=\bar \lambda^U=\frac{1}{2}$, $\bar \mu_t=0$, $\forall t\in [0, 1)$, and $\bar \mu_1=1$. It is easy to see that $(\bar y, \bar\lambda^L,\bar \lambda^U, \bar \mu)\in \mathcal{F}_{MW}$. However, for $\bar x=0\in\mathcal{F}$, we have
		\begin{align*}
		f(\bar x)=[0, 1]&<^s_{LU} \mathcal{L}(\bar y, \bar\lambda^L, \bar \lambda^U, \bar{\mu})
		\\
		&=\bigg[\frac{1}{3}, \frac{4}{3}\bigg]-\bigg[\frac{1}{5}, \frac{1}{4}\bigg]=\bigg[\frac{1}{12}, \frac{17}{15}\bigg].
		\end{align*}
		This means that the conclusion of Theorem \ref{weak-dual} is no longer true. The reason is that the $\mathcal{E}$-pseudo-quasi generalized convexity of $(f,g_T)$ on $\Omega$ at $\bar y$ has been violated. 
		\begin{center}
			\begin{figure}[htp]
				\begin{center}
					\begin{tikzpicture}[]
					\begin{axis}[axis lines=center,
					axis on top,
					xtick=\empty,
					ytick=\empty,
					xrange=-5:5,
					yrange=-5:5
					]
					\addplot[domain=-5:5,y domain=-5:5,red]
					{x^3/3};
					\addlegendentry{$f^L(x)$}
					\addplot[domain=-5:5,y domain=-5:5,blue]
					{x^3/3+1};
					\addlegendentry{$f^U(x)$}
					\end{axis}
					\end{tikzpicture}
					\caption{Plot of $f^L$ and $f^U$ in Example \ref{Example-3}.}\label{Fig.3}
				\end{center}
			\end{figure}
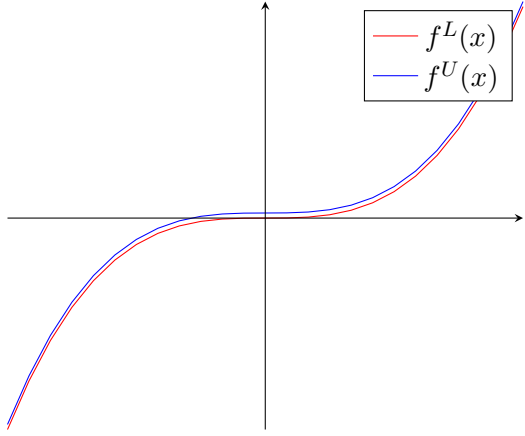
		\end{center}
	}
\end{example}  

Next we present a theorem that formulates strong duality relations between the primal problem \eqref{problem} and the dual problem \eqref{Dual-problem}.
\begin{theorem}[$\mathcal{E}$-strong duality] Let $\bar x$ be a type-2 $\mathcal{E}$-quasi-weakly Pareto solution of \eqref{problem} and assume that the \eqref{LCQ} holds at this point. Then there exist $\bar \lambda^L, \bar\lambda^U\in \mathbb{R}^m_+$, and $\bar \mu\in A(\bar x)$ such that $(\bar x, \bar \lambda^L, \bar \lambda^U, \bar \mu)\in \mathcal{F}_{MW}$, $f(\bar x)=\mathcal{L}(\bar x, \bar \lambda^L, \bar\lambda^U, \bar \mu)$. Furthermore,
	\begin{enumerate}[\rm(i)]
		\item If $(f, g_T)$ is $\mathcal{E}$-pseudo-quasi generalized convex on $\Omega$ at $\bar x$, then $(\bar x, \bar \lambda^L, \bar \lambda^U, \bar \mu)$ is a type-2 $\mathcal{E}$-quasi weakly Pareto solution of \eqref{Dual-problem}.
		\item If $(f, g_T)$ is strictly $\mathcal{E}$-pseudo-quasi generalized convex on $\Omega$ at $\bar x$, then $(\bar x, \bar \lambda^L, \bar \lambda^U, \bar \mu)$ is a type-1 $\mathcal{E}$-quasi  Pareto solution of \eqref{Dual-problem}.
	\end{enumerate}
\end{theorem}
\begin{proof} By Theorem \ref{Necessary-Theorem}, there exist $\bar\lambda^L, \bar\lambda^U\in \mathbb{R}^m_+$, and $\bar\mu\in A(\bar x)$ satisfying $(\bar x, \bar\lambda^L, \bar\lambda^U, \bar\mu)\in \mathcal{F}_{MW}$. Clearly, 
	$$f(\bar x)=\mathcal{L}(\bar x, \bar\lambda^L, \bar\lambda^U, \bar\mu).$$
	(i) If $(f, g_T)$ is $\mathcal{E}$-pseudo-quasi generalized convex on $\Omega$ at $\bar x$, then, by invoking now (i) of Theorem \ref{weak-dual}, we obtain
	\begin{equation*}
	f(\bar x)=\mathcal{L}(\bar x, \bar\lambda^L, \bar\lambda^U, \bar\mu)\nprec^s_{LU} L(y,\lambda^L,\lambda^U, \mu)-\mathcal{E}\|\bar x-y\|
	\end{equation*}    
	for all $(y,\lambda^L,\lambda^U, \mu)\in\mathcal{F}_{MW}$. Therefore, $(\bar x, \bar\lambda^L, \bar\lambda^U, \bar\mu)$ is a type-2 $\mathcal{E}$-quasi weakly Pareto solution of \eqref{Dual-problem}. 
	
	The proof of (ii) is similar to that of (i) by using the strictly $\mathcal{E}$-pseudo-quasi generalized convexity of $(f, g_T)$ on $\Omega$ at $\bar x$ instead of the $\mathcal{E}$-pseudo-quasi generalized convexity of $(f, g_T)$ on $\Omega$ at the corresponding point. 
\end{proof}
We close this section by presenting converse-like duality relations for approximate quasi Pareto solutions between the primal problem \eqref{problem}  and the dual problem  \eqref{Dual-problem}.
\begin{theorem}[Converse-like duality] Let $(\bar x, \bar\lambda^L, \bar\lambda^U, \bar\mu)\in \mathcal{F}_{MW}$.
	\begin{enumerate}[\rm(i)]
		\item If $\bar x\in\mathcal{F}$ and $(f, g_T)$ is $\mathcal{E}$-pseudo-quasi generalized convex on $\Omega$ at $\bar x$, then $\bar x$ is a type-2 $\mathcal{E}$-quasi weakly Pareto solution of \eqref{problem}. 
		\item If $\bar x\in\mathcal{F}$ and $(f, g_T)$ is strictly $\mathcal{E}$-pseudo-quasi generalized convex on $\Omega$ at $\bar x$, then $\bar x$ is a type-1 $\mathcal{E}$-quasi  Pareto solution of \eqref{problem}.
	\end{enumerate}	
\end{theorem}
\begin{proof} (i) Since $(\bar x, \bar\lambda^L, \bar\lambda^U, \bar\mu)\in \mathcal{F}_{MW}$, we have
	\begin{equation*}
	0\in \sum_{i\in I}[\bar\lambda_i^L\partial f_i^L(\bar x)+\bar\lambda_i^U\partial f_i^U(\bar x)]+\sum_{t\in T}\bar \mu_t\partial g_t(\bar x)+\sum_{i\in I}(\bar \lambda_i^L\epsilon_i^U+\bar \lambda_i^U\epsilon_i^L)\mathbb{B}_{\mathbb{R}^n}+N(\bar x; \Omega),
	\end{equation*}	
	$\sum_{i\in I}(\bar\lambda_i^L+\bar\lambda_i^U)=1$, and 
	\begin{equation}\label{equ:9}
	\bar \mu_t g_t(\bar x)\geq 0, \ \ \forall t\in T.
	\end{equation}
	It follows from  $\bar x\in\mathcal{F}$ and \eqref{equ:9} that $\bar \mu_t g_t(\bar x)= 0$ for all $t\in T$, i.e., $\bar \mu\in A(\bar x)$. By the  $\mathcal{E}$-pseudo-quasi generalized convexity of $(f, g_T)$ on $\Omega$ at $\bar x$ and  Theorem \ref{Sufficient-Theorem}(i), $\bar x$ is a type-2 $\mathcal{E}$-quasi weakly Pareto solution of \eqref{problem}. 
	
	(ii) The proof of (ii) is quite similar to that of (i) by using the strictly $\mathcal{E}$-pseudo-quasi generalized convexity of $(f, g_T)$  and Theorem \ref{Sufficient-Theorem}(ii), so it is omitted.
\end{proof}


\section{Conclusions}\label{Conclusions}
In this paper, we focus for the first time on studying approximate solutions of semi-infinite interval-valued vector optimization problems. By employing    the Mordukhovich/limiting subdifferential and the Mordukhovich/limiting normal cone and introducing some types of approximate pseudo-quasi generalized convexity of a family of locally Lipschitz functions, we establish optimality conditions of KKT-type and duality relations for proposed approximate solutions of the considered problem. Observe that the class of approximate pseudo-quasi generalized convex functions is properly wider than the class of convex functions  in the sense of \cite{Chuong-14,Chuong-Kim-16,Rockafellar-70}, our results generalize  and improve some  existing results. Furthermore, the model of problems of the form \eqref{problem} covers the one of cone-constrained convex vector optimization problems and semidefinite vector optimization problems,  and so our approach can be developed to study these problems. We aim to investigate this problem in future work.       

\section*{Acknowledgments}{}
We would like to thank the anonymous referee for insightful comments and valuable suggestions.

\section*{Disclosure statement}
No potential conflict of interest was reported by the authors.

\section*{Funding}
This research is funded by Hanoi Pedagogical University 2 under grant number HPU2.UT-2021.15. 


\begin{thebibliography}{99}
	\bibitem{Ben-Tal-Nemirovski-98}
	{ Ben-Tal  A, Nemirovski A.} {Robust convex optimization}. {Math Oper Res}. 1998;23:769--805. 
	
	\bibitem{Ben-Tal-Nemirovski-08}
	{ Ben-Tal  A, Nemirovski A.} {A selected topics in robust convex optimization}. {Math Program.} 2008;112:125--158. 
	
	\bibitem{Ben-Tal-Nemirovski-09}
	{ Ben-Tal A, Ghaoui LE, Nemirovski A.} 
	{Robust optimization}. Princeton(NJ): Princeton University Press; 2009.
	\bibitem{Rahimi}
	Rahimi M, Soleimani-Damaneh M. Robustness in deterministic vector optimization. J Optim  Theory Appl. 2018;179:137--162.
	\bibitem{Slowinski-98}
	{ S{\L}owi\'nski R.} {Fuzzy sets in decision analysis, operations research and statistics}. Boston (MA): Kluwer Academic Publishers;   1998.
	\bibitem{Wu-09-c}
	Wu HC.  Duality theory for optimization problems with interval-valued objective functions. J Optim
	Theory Appl. 2010;144:615--628.
	\bibitem{Chalco-Cano et.al.-13}
	{Chalco-Cano Y,  Lodwick WA, Rufian-Lizana A.}  {Optimality conditions of type KKT for optimization problem with interval-valued objective function via generalized derivative}. {Fuzzy Optim Decis Mak.} 2013;12:305--322.
	\bibitem{Ishibuchi-Tanaka-90}
	{Ishibuchi H,  Tanaka H.} {Multiobjective programming in optimization of the interval objective function}. {European J Oper Res.} 1990;48:219--225.  
	
	
	
	\bibitem{Jennane}
	Jennane M, Kalmoun EM, Lafh L. Optimality conditions for nonsmooth interval-valued and multiobjective semi-infinite programming. RAIRO-Oper Res. 2021;55:1--11.
	\bibitem{Kumar}
	Kumar P, Sharma B,  Dagar J. Interval-valued programming problem with infinite constraints.  J Oper Res Soc China. 2018;6:611--626.
	\bibitem{Osuna-Gomez-17}
	Osuna-G\'omez  R, Hern\'adez-Jim\'enez B, Chalco-Cano Y, et al. New efficiency conditions 	for multiobjective interval-valued programming problems. Inf Sci. 2017;420:235--248.
	\bibitem{Qian}
	Qian X, Wang KR, Li XB. Solving vector interval-valued optimization problems with infinite interval constraints via integral-type penalty function. Optimization. 2021. https://doi.org/10.1080/02331934.2021.1906872.
	\bibitem{Singh-Dar-15}
	{Singh AD, Dar BA.} {Optimality conditions in multiobjective programming problems with interval valued objective functions}. {Control Cybern.}  2015;44:19--45.  
	
	\bibitem{Singh-Dar-Kim-16}
	{Singh D,  Dar  BA,  Kim  DS.} {KKT optimality conditions in interval valued multiobjective programming with generalized differentiable functions}. {European J Oper Res.} 2016;254:29--39.
	
	\bibitem{Singh-Dar-Kim-19}
	{Singh D,  Dar  BA,  Kim  DS.} {Sufficiency and duality in non-smooth interval valued programming problems}. {J Ind Manag Optim.}  2019;15:647--665.
	\bibitem{Su-2020}
	Su TV, Dinh DH. Duality results for interval-valued pseudoconvex optimization problem with equilibrium constraints with applications. Comput Appl Math. 2020;39:1--24.
	
	\bibitem{Tung-2019}
	{Tung LT.} {Karush--Kuhn--Tucker optimality conditions and duality for semi-infinite programming with multiple interval-valued objective functions}. {J Nonlinear Funct Anal.}  2019;2019:1--21. 
	
	\bibitem{Tung-2019b}
	{Tung LT.} {Karush--Kuhn--Tucker optimality conditions and duality for convex semi-infinite programming with multiple interval-valued objective functions}. {J Appl Math Comput.} 2020;62:67--91.
	\bibitem{Wu-07}
	{ Wu HC.}  {The Karush--Kuhn--Tuker optimality conditions in an optimization problem with interval valued objective functions}. {Eur J Oper Res.} 2007;176:46--59.
	
	\bibitem{Wu-09}
	{Wu HC.}  {The Karush--Kuhn--Tucker optimality conditions in multiobjective programming problems with interval-valued objective functions}. {Eur J Oper Res.}  2009;196:49--60.
	
	\bibitem{Wu-09-b}
	{Wu HC.}  {The Karush--Kuhn--Tucker optimality conditions for multi-objective programming problems with fuzzy-valued objective functions}. {Fuzzy Optim Decis Mak.} 2009;8:1--28. 
	\bibitem{Tuyen-2021}
	Tuyen NV. Approximate solutions of interval-valued optimization problems. Investigaci\'on Oper.  2021;42:223--237. 
	\bibitem{Jahn-2011}  Jahn J. {Vector optimization: theory, applications and extensions}. Berlin: Springer; 2011.
	\bibitem{Luc-89} Luc DT. Theory of vector optimization. Berlin: Springer; 1989.
	\bibitem{Bao-et al}
	{Bao TQ, Eichfelder G, Soleimani B, et al.} Ekeland's variational principle for vector optimization with variable ordering structure. {J Convex Anal.} 2017;24:393--415. 
	\bibitem{Chen13}
	Chen Z. Characterizations of solution sets for parametric multiobjective optimization problems. Appl. Anal. 2013;92:2457--2467.
	\bibitem{Chuong-Kim-16}
	Chuong TD,  Kim DS. {Approximate solutions of multiobjective optimization problems}. {Positivity}. 2016;20:187--207.
	\bibitem{Kim-Son-18}
	Kim DS, Son TQ. {An approach to $\epsilon$-duality theorems for nonconvex semi-infinite multiobjective optimization problems}.  Taiwanese J Math. 2018;22:1261--1287. 
	\bibitem{Loridan-84}
	{ Loridan P.}  {$\epsilon$-solutions in vector minimization problems}. {J Optim Theory Appl.}  1984;43:265--276.
	\bibitem{Son-Tuyen-Wen-19}
	{Son TQ,   Tuyen  NV, Wen  CF.} {Optimality conditions for approximate Pareto solutions of a nonsmooth vector optimization problem with an infinite number of constraints}. {Acta Math Vietnam.} 2020;45:435--448.  
	\bibitem{TXS-20}
	{ Tuyen NV,  Xiao YB, Son TQ.}  {On approximate KKT optimality conditions for cone-constrained vector optimization problems}. {J Nonlinear Convex Anal.} 2020;21:105--117.
	\bibitem{mor06}
	{Mordukhovich BS.} {Variational analysis and generalized differentiation, I: basic theory, II: applications.} Grundlehren series 	(fundamental principles of mathematical sciences). Vol. 330. Berlin: Springer; 2006.	
	\bibitem{Alefeld}
	Alefeld G,  Herzberger J. {Introduction to interval computations}. New-York (NY): Academic Press;  1983.  
	\bibitem{Moore-1966}
	{ Moore RE.} {Interval analysis}. Englewood Cliffs (NJ): Prentice-Hall; 1966.
	
	\bibitem{Moore-1979}
	{Moore RE.} {Method and applications of interval analysis}. Philadelphia (PA): SIAM; 1979.
	\bibitem{Chuong-09}  Chuong TD, Huy NQ,   Yao JC.  {Subdifferentials of marginal functions in semi-infinite programming}. SIAM J Optim. 2009;20:1462--1477.
	
	\bibitem{Chuong-14}  Chuong TD,  Kim DS.  {Nonsmooth semi-infinite multiobjective optimization problems}. J Optim
	Theory Appl. 2014;160:748--762.  
	
		\bibitem{Jiao-Liguo-21}
	Jiao LG, Kim  DS, Zhou Y. Quasi $\epsilon $-solutions in a semi-infinite programming problem with locally Lipschitz data.  Optim Lett. 2021;15:1759--1772.
	\bibitem{Ioffe79} Ioffe AD, Tikhomirov VM. Theory of extremal problems. Amsterdam: North-Holland; 1979.  
	\bibitem{Fakhar-19}
	Fakhar M, Mahyarinia  MR, Zafarani J. On approximate solutions for nonsmooth robust multiobjective optimization problems. Optimization.  2019. https://doi.org/10.1080/02331934.2019.1579212
	
	\bibitem{Rockafellar-70} 	{Rockafellar   RT.} {Convex analysis}. Princeton (NJ): Princeton University Press; 1970.
	
	
	
	
	
	
	
	
		
	
	
	
	
	
	
	
	
	
\end{thebibliography}
\end{document}